\documentclass[12pt]{amsart}
\usepackage{amsmath, amssymb, amsthm} 
\usepackage{hyperref} 
\usepackage{geometry} 
\usepackage{color}
\DeclareMathOperator{\tr}{tr}
\newtheorem{theorem}{Theorem}[section]
\newtheorem{lemma}[theorem]{Lemma}

\newtheorem{proposition}[theorem]{Proposition}

\theoremstyle{definition}
\newtheorem{definition}[theorem]{Definition}
\newtheorem{example}[theorem]{Example}
\newtheorem{notation}[theorem]{Notation}
\theoremstyle{remark}
\newtheorem{remark}[theorem]{Remark}

\DeclareMathOperator{\diag}{diag}
\DeclareMathOperator{\sech}{sech}
\DeclareMathOperator{\Hess}{Hess}
\DeclareMathOperator{\Gr}{Gr}
\DeclareMathOperator{\St}{St}

\title{Constructing entire minimal graphs by evolving planes}


\author{Chung-Jun Tsai}
\address{Department of Mathematics, National Taiwan University, and National Center for Theoretical Sciences, Math Division, Taipei 10617, Taiwan}
\email{cjtsai@ntu.edu.tw}

\author{Mao-Pei Tsui}
\address{Department of Mathematics, National Taiwan University, and National Center for Theoretical Sciences, Math Division, Taipei 10617, Taiwan}
\email{maopei@math.ntu.edu.tw}

\author{Jingbo Wan}
\address{Laboratoire Jacques-Louis Lions de Sorbonne Universit\'e, 4 place Jussieu, Paris 75005, France}
\email{jingbo.wan@sorbonne-universite.fr}

\author{Mu-Tao Wang}
\address{Department of Mathematics, Columbia University, New York, NY 10027, USA}
\email{mtwang@math.columbia.edu}

\thanks{C.-J.~Tsai is supported in part by the National Science and Technology Council grant 112-2628-M-002-004-MY4.  M.-P.~Tsui is supported in part by the National Science and Technology Council grant 112-2115-M-002-015-MY3 and 113-2918-I-002-004. M.-T. ~Wang is supported in part by the National Science Foundation under Grants DMS-2104212 and DMS-2404945. This work was supported by grants from the Simons Foundation [SFI-MPS-SFM-00006056 and MPS-TSM-00007411, M.-T.~W.]. Part of this work was carried out when M.-T.~Wang was visiting the Institute of Mathematics, Academia Sinica. The authors would like to thank Professor Robert Bryant for bringing to their attention the correspondence between entire austere graphs and entire special Lagrangian graphs, which is central to the formulation of  Theorem \ref{thm:entire_SPL}.} 

\begin{document}

\begin{abstract}
We introduce an evolving-plane ansatz for the explicit construction of entire minimal graphs of dimension $n$ ($n\geq 3$) and codimension $m$ ($m\geq 2$), for any odd integer $n$. Under this ansatz, the minimal surface system reduces to the geodesic equation on the Grassmannian in affine coordinates. Geometrically, this equation dictates how the slope of an $(n-1)$ plane evolves as it sweeps out a minimal graph. This framework yields a large family of explicit entire minimal graphs of odd dimension $n$ and arbitrary codimension $m$.
For each entire minimal graph, its conormal bundle gives rise to an entire special Lagrangian graph in $\mathbb{C}^{n+m}$.

\end{abstract}
\maketitle
\section{Introduction}

Entire solutions---those defined on all of Euclidean space---play a central role in the study of elliptic partial differential equations. 
The classical Liouville theorem and Bernstein theorem characterize entire solutions of the Laplace equation and of the minimal surface equation, respectively. A landmark result in minimal surface theory is the well-known Bombieri–De~Giorgi–Giusti \cite{BDG} construction of a non-trivial entire solution to the minimal surface equation on $\mathbb{R}^8$. A broad generalization of the minimal surface equation is the minimal surface system which governs minimal submanifolds of higher codimensions, but remains comparatively underdeveloped.
In this work, we introduce an evolving-plane ansatz that reduces the minimal surface system to the geodesic equation on the Grassmannian and produces explicit entire minimal graphs of dimension $n$ and codimension $m$ where $n\geq3$ is odd and $m\geq 2$. Although entire minimal graphs are known in even dimensions, such as those arising from entire holomorphic functions on complex Euclidean space, this work provides what seems to be  the first general construction of entire minimal graphs of any odd dimension. 

We begin with the definition of minimal graphs: 

\begin{definition}
Let $n, m \geq 1$ be integers. Let $\Omega\subset\mathbb{R}^{n}$ be a domain and 
\[
(f^\alpha)_{\alpha=1,\ldots,m}:\Omega\to\mathbb{R}^m
\]
be a smooth vector-valued function. The \emph{graph}
\[
\Gamma \;=\;\bigl\{(x,f^\alpha(x)) : x\in\Omega \bigr\}
\;\subset\;\mathbb{R}^{n}\times\mathbb{R}^m=\mathbb{R}^{n+m},
\]
is called a \emph{minimal graph of dimension $n$} and \emph{codimension $m$}
if its mean curvature vanishes identically. 
\end{definition}
The equation can be derived by the first variation formula for the volume functional. In the $m=1$ case, the minimal surface equation for an unknown function $f$ defined on $\mathbb{R}^n$ is:
\[
\sum_{i}\partial_{i}\Bigl(\frac{\partial_{i}f}{\sqrt{1 + |\nabla f|^2}}\Bigr)=0 .
\]

In the $m=2$ case, the {\em minimal surface system} for two unknown functions $f, g$ defined on $\mathbb{R}^n$ is:
\[
\begin{cases}
\displaystyle
\sum_{i}\partial_{i}\Bigl(
\frac{(1 + |\nabla g|^2)\,\partial_{i}f \;-\; (\nabla f\cdot\nabla g)\,\partial_{i}g}
{\sqrt{(1 + |\nabla f|^2)\,(1 + |\nabla g|^2) - (\nabla f\cdot\nabla g)^2}}
\Bigr)=0,\\[8pt]
\displaystyle
\sum_{i}\partial_{i}\Bigl(
\frac{(1 + |\nabla f|^2)\,\partial_{i}g \;-\; (\nabla f\cdot\nabla g)\,\partial_{i}f}
{\sqrt{(1 + |\nabla f|^2)\,(1 + |\nabla g|^2) - (\nabla f\cdot\nabla g)^2}}
\Bigr)=0.
\end{cases}
\]

We shall not rely on this particular form of the minimal surface system. However, one sees how the equations are nonlinearly coupled in a complex way. In their classical paper, Lawson and Osserman \cite{LO77} demonstrated the non-existence, non-uniqueness and irregularity of solutions to the minimal surface system. 

A different characterization of the solutions to the minimal surface system \cite{LO77} will be adopted and we will be solving the following equation:
\begin{proposition}\label{prop:minimal-basic}
Let $(f^\alpha)_{\alpha=1,\ldots,m}\colon \Omega\to\mathbb{R}^m$ be a vector-valued $C^2$ function on a domain $\Omega\subset\mathbb{R}^{n}$. Then, its graph is a minimal graph if and only if
\begin{equation}\label{eqn:mss}
g^{ij}\,\partial_i\partial_j f^\alpha \;=\; 0,
\qquad \alpha = 1,\dots,m,
\end{equation}
where $(g^{ij})$ denotes the inverse matrix of
\[
g_{ij} \;=\; \delta_{ij} \;+\;\sum_{\beta=1}^m \partial_i f^\beta\,\partial_j f^\beta ,
\]
the induced metric on the graph of $f^\alpha$.
In this case, the expression $g^{ij}\,\partial_i\partial_j f^\alpha$ is the same as $\Delta_g f^\alpha$, the Laplacian of $f^\alpha$ with respect to the induced metric $g=(g_{ij})$.
\end{proposition}
From now on, we denote the standard coordinates on $\mathbb{R}^n$ by $(x^1,\ldots, x^{n-1}, t)$.
\begin{definition}The above $(f^\alpha)_{\alpha=1,\ldots,m}$ is said to satisfy the {\em evolving-plane ansatz} if there exist $(n-1)\times m$ functions $\{z_i^\alpha(t)\}_{\alpha=1,\ldots,m; \ i=1,\ldots,n-1}$ depending on the variable $t$ such that in terms of coordinates $(x^1,\ldots,x^{n-1},t)$ on $\mathbb{R}^n$,
\begin{equation}\label{eqn:form}
f^\alpha(x^1,\ldots,x^{n-1},t)=\sum_{i=1}^{n-1} z^\alpha_i (t)x^i.
\end{equation}
\end{definition}
It is convenient to consider $t$ as a time variable.  At $t =0$, the graph of the map ${f}^\alpha$ is an $n-1$ dimensional affine subspace of $\mathbb{R}^{n+m}$. As $t$ progresses, the ansatz evolves the affine subspace and generates an $n$-dimensional submanifold. Note that the evolution of $z_i^\alpha$ governs how the plane changes its slope.  It is convenient to introduce the following definition.
\begin{definition}The $(n-1)\times m$ matrix $Z(t)=\left[z_i^\alpha(t)\right] $ is called the {\em slope matrix}.
\end{definition}

\begin{theorem} \label{mss_ode} Suppose that each $\alpha=1,\ldots, m$, \(f^\alpha\) is a $C^2$ function on a domain $\Omega\subset \mathbb{R}^n$ and of the form \eqref{eqn:form}. Then $(f^\alpha)_{\alpha=1,\ldots,m}$ satisfies the minimal surface system \eqref{eqn:mss} if and only if the slope matrix $Z(t)$ satisfies the following second order ODE system: 
\begin{equation}\label{eqn:slope_evolution}
\ddot{Z}-2\dot{Z} Z^\top (I_{n-1}+ZZ^\top )^{-1}\dot{Z}=0.
\end{equation}
\end{theorem}
Equation \eqref{eqn:slope_evolution} is exactly the geodesic equation on the Grassmannian \footnote{This equation first appeared explicitly in Y.-C. Wong's 1967 paper (Theorem 5 of \cite{Wong}). Later, it was more widely recognized and used in the computational geometry/optimization literature; see, for example, \cite{BZA, EAS1998}.} in affine coordinates \cite{Wong, MF1996}, see the next section.

\begin{example}
    Let $r$ be an integer with $0 \leq r \leq \min(n-1,m)$, and $\lambda_i$, $i=1,\ldots, r$ be given with $0 < \lambda_1 \leq \cdots \leq \lambda_r$, we denote the diagonal matrix \(\diag(\lambda_1, \ldots, \lambda_r)\) by,
\[
    \Lambda = \diag(\lambda_1, \ldots, \lambda_r). 
\]

A basic family of solutions to \eqref{eqn:slope_evolution} is given by the initial conditions 
\[
Z(0) = O_{(n-1) \times m}, \quad 
\dot{Z}(0) = 
\begin{bmatrix}
\Lambda &0\\[2pt]
0&0
\end{bmatrix}.
\]
Denoting $\tan(\Lambda t)=\diag(\tan(\lambda_1 t),\ldots,\tan(\lambda_r t))$, the solution to \eqref{eqn:slope_evolution} with these initial conditions is 
\[
Z(t) = 
\begin{bmatrix}
\tan(\Lambda t)& 0\\0& 0
\end{bmatrix}.
\]

\end{example}

This family corresponds to the generalized helicoids studied by Bryant \cite{Bryant} and Barbosa-Dajczer-Jorge \cite{BDJ}. These solutions blow up along the \(t\)-axis at \(t = \frac{\pi}{2\lambda_r}\) and hence are not entire graphs.

In this work, we identify a large family of initial conditions for \eqref{eqn:slope_evolution} that produce entire solutions of the minimal surface system \eqref{eqn:mss}.

\begin{notation}
Let
\[
\Lambda = \diag(\lambda_1, \ldots, \lambda_r), 
\quad 0 \leq r \leq \min(n-1,m), 
\quad 0 < \lambda_1 \leq \cdots \leq \lambda_r,
\]
and define the block matrices
\[
\widetilde{\Lambda} 
= \begin{bmatrix}
\Lambda  & 0 \\[2pt]
0 & 0
\end{bmatrix}
\quad\in \mathbb{R}^{(n-1)\times m},
\]

\[
\cos(\widetilde{\Lambda} t)
= \begin{bmatrix}
\cos(\Lambda t) & 0 \\[2pt]
0 & I_{m-r}
\end{bmatrix}
\quad\in \mathbb{R}^{m\times m},
\]
\[
\sin(\widetilde{\Lambda} t)
= \begin{bmatrix}
\sin(\Lambda t) & 0 \\[2pt]
0 & 0
\end{bmatrix}
\quad\in \mathbb{R}^{(n-1)\times m}.
\]
\end{notation}

\begin{theorem}\label{thm:entire_ODE}
 Let \(B\) be an \( (n-1)\times m\) matrix.
Suppose that
\begin{equation}\label{positivity}
\det\bigl(\cos(\widetilde{\Lambda} t)+B^\top \sin(\widetilde{\Lambda} t)\,\bigr)>0
\quad\text{for all }t.
\end{equation}
Then the unique solution $Z(t)$ of \eqref{eqn:slope_evolution} with the following initial conditions
\[
Z(0)=-B \quad\text{and}\quad \dot{Z}(0)= (I + B B^\top )^{1/2}\,\widetilde{\Lambda}\,(I + B^\top B)^{1/2}.
\] exists for all $t\in\mathbb{R}$.
\end{theorem}

$Z(t)$ is given explicitly in \eqref{Z-solution}. When $n$ is an odd integer, there exists a rich family of pairs ($\widetilde{\Lambda}$, $B$) satisfying \eqref{positivity}, which in turn yields the following family of entire solutions of the minimal surface system. 

\begin{theorem}\label{thm:entire_odd}
Suppose $n\geq 3$ is an odd integer and $m\geq 2$. Then there exist entire solutions $(f^\alpha)_{\alpha=1,\ldots,m}: \mathbb{R}^n\rightarrow \mathbb{R}^m$ of the minimal surface system \eqref{eqn:mss} of the form $f^\alpha=\sum_{i=1}^{n-1} z^\alpha_i (t)x^i$ \eqref{eqn:form} with $Z(t)=\left[z_i^\alpha(t)\right]$ given explicitly.
\end{theorem}

The simplest example produced by Theorem \ref{thm:entire_odd} is the following. 

\begin{example} ($n=3, m=2$) Denote \begin{equation}{I}=\begin{bmatrix}
1 & 0\\
0 & 1
\end{bmatrix}, \qquad
{J}=\begin{bmatrix}
0 & 1\\
-1 & 0
\end{bmatrix}.
\end{equation} We take $\widetilde{\Lambda}=\Lambda=\frac{1}{2}I$ (thus $r=2$), $B=J$ and the initial conditions become ${Z}(0)=-J$ and $\dot{{Z}}(0)=I$.
The solution is found to be 
\[{Z}(t)=\begin{bmatrix}
\sin t & -\cos t\\
\cos t & \sin t 
\end{bmatrix}.\]
The corresponding entire minimal graph is given by the embedding:
\[(x^1, x^2, t)\rightarrow (x^1, x^2, t, x^1\sin t-x^2\cos t, x^1\cos t+x^2 \sin t).\]

We denote the coordinates on the ambient $\mathbb{R}^5$ by $(x^1, x^2, t, y^1, y^2)$. The intersection of the minimal graph with an $\mathbb{R}^4$ of constant $t$ is a plane that lies in the minimal cone
\[ (x^1)^2+(x^2)^2=(y^1)^2+(y^2)^2 \]
over the Clifford torus.

\end{example}
The following family of examples yields entire minimal graphs of any odd dimension. 
\begin{example}

Let $p$ be a positive integer, $n=2p+1$, and $m=2p$. Let $\Lambda=I$ be the $2p\times 2p $ identity matrix, and $B$ be any $2p\times 2p$ real matrix that has no real eigenvalues. Then 
\[ Z(t)= 
 N\bigl(-\cos t \cdot B + \sin t\cdot I\bigr)\,
 \bigl(\cos t \cdot I + \sin t\cdot  B^\top \bigr)^{-1} M^{-1} ,
\]
where 
\[
M = (I + B^\top B)^{-1/2} \quad\text{ and }\quad
N = (I + B B^\top )^{-1/2} ,
\] is an entire solution of \eqref{eqn:slope_evolution}. The corresponding entire minimal graph is of dimension $n=2p+1$ and codimension $m=2p$.
\end{example}

By the $O(m)$ symmetry of the minimal surface system, additional solutions can be generated by acting on a given solution $Z(t)$ with any $R\in O(m)$ to produce the new solution $R Z(t)$. Moreover, the pair $(\widetilde{\Lambda}, B)$ admits an additional symmetry given by a signature matrix (a matrix with diagonal entries $\pm 1$) which is subsumed by the $O(m)$ symmetry. This latter symmetry permits us, without loss of generality, to restrict our attention to $\widetilde{\Lambda}$ with positive diagonal entries.

Entire minimal graphs of even dimensions can be obtained from entire holomorphic or anti-holomorphic functions defined on even-dimensional Euclidean spaces $\mathbb{R}^n$. When $n=2$, Osserman (Section 5 in \cite{Os86} ) constructed more entire solutions that are neither holomorphic nor anti-holomorphic by his isothermal parametrization theorem.  
There are also constructions of minimal graphs by resolving the Lawson-Osserman cones and their generalizations (see \cite{DY} and \cite{XYZ}). 

A theorem of Micallef (see \cite{84Mi}) asserts that any stable, entire, two-dimensional minimal graph in $\mathbb{R}^4$ must be holomorphic with respect to some orthogonal complex structure on $\mathbb{R}^4$ (see Corollary 5.1 on page 68 of \cite{84Mi}). Consequently, the non-holomorphic/anti-holomorphic examples constructed by Osserman in \cite{Os86} must be unstable. It is therefore natural to investigate the stability of  our new examples, which, in contrast, are entire minimal graphs defined on odd-dimensional Euclidean spaces. In this connection, we recall that a classical theorem of do Carmo-Peng \cite{dP} and independently Fischer-Colbrie-Schoen \cite{80FCS} implies the classical helicoid in $\mathbb{R}^3$ is unstable.

Each minimal submanifold constructed under this ansatz is a generalized helicoid and therefore austere; consequently, its conormal bundle is a special Lagrangian submanifold by Harvey-Lawson \cite{HL82}. Robert Bryant pointed out to us that entire austere graphs correspond to entire special Lagrangian graphs. Thus, we obtain a large family of new entire special Lagrangian graphs. These appear to differ from those constructed in \cite{TTW}, although both arise from the ``variation of quadratic polynomials" ansatz. The Lagrangian phase is zero for the examples arising from entire minimal graphs. \footnote{In forthcoming work, we will investigate the most general class of entire solutions produced by the variation of quadratic polynomials ansatz.} These results are summarized in the following theorem. 
\begin{theorem}\label{thm:entire_SPL}
Suppose $n\geq 3$ is an odd integer and $m\geq 2$. Let $(x^1, \cdots, x^{n-1}, t, u_1,\cdots, u_m)$ denote coordinates on $\mathbb{R}^{n+m}$ and consider the function 
\[F=\sum_{i=1}^{n-1}\sum_{\alpha=1}^m z_i^\alpha(t) x^i u_\alpha,\] where $Z(t)=[z_i^\alpha (t)]$ is a solution of equation \eqref{eqn:slope_evolution} that exists for $t\in (-\infty, \infty)$ as in Theorem \ref{thm:entire_ODE}. Then the Lagrangian graph defined by $\nabla F$ in $\mathbb{C}^{n+m}$ is an entire special Lagrangian graph with phase zero.  
\end{theorem}

\subsection{Scope of the Paper}
This paper is organized as follows.
\begin{itemize}
 \item In Section \ref{sec:Ansatz-MSS}, we introduce the evolving-plane ansatz and show that the minimal surface system is reduced to the Grassmannian geodesic equation under this ansatz and prove Theorem \ref{mss_ode}.
 \item In Section \ref{sec:geodesic}, we discuss the Grassmannian geodesic in different models of Grassmannian and show how the geodesic equation transforms from one model to another. 
 
 \item In Section \ref{sec:entire_ODE}, we prove Theorem \ref{thm:entire_ODE}.
 \item In Section \ref{sec:entire_odd}, we prove Theorem \ref{thm:entire_odd} by identifying $B$ and $\widetilde{\Lambda}$ that satisfy the assumptions of Theorem \ref{thm:entire_ODE}. 

 \item In Section \ref{sec:entire_SPL}, we prove Thereom \ref{thm:entire_SPL}.
\item In Appendix \ref{app:Lorentzian}, we generalize the construction to semi-Euclidean spaces of signature $(n-1, m+1)$ and provide examples of entire graphs of vanishing mean curvature over the Minkowski space $\mathbb{R}^{n-1,1}$ in $\mathbb{R}^{n-1, m+1}$.
\end{itemize}
 
\section{The Evolving-Plane Ansatz and Minimal Surface System}\label{sec:Ansatz-MSS}

This section is devoted to the proof of Theorem \ref{mss_ode}.
Let $n\geq 2, m \geq 1$ be integers. Given $(n-1)\times m$ functions $\{z_i^\alpha(t)\}_{\alpha=1,\ldots,m; \ i=1,\ldots,n-1}$ depending on the variable $t$, consider $f^\alpha(\vec{x},t)= \sum\limits_{i=1}^{n-1} x^i z^\alpha_i(t)$. The graph of $f^\alpha$ defines an embedding 
\begin{align*}
    \begin{array}{cccl}
         \Phi: &\mathbb R^{n} &\rightarrow &\mathbb R^{n+m} \\
         &(x^1, \ldots, x^{n-1}, t) &\mapsto &\bigl(x^1, \ldots, x^{n-1}, t, \sum_{i=1}^{n-1} x^i z^1_i(t), \ldots, \sum_{i=1}^{n-1} x^i z^m_i(t)\bigr) .
    \end{array}
\end{align*}

The first step is to compute the induced metric. A basis of the tangent space of the embedding 
is given by
\begin{align*}
    \Phi_{x^j}&=\bigl(\vec{e}_j,0,z_j^1(t),\ldots,z_j^m(t)\bigr),\\
    \Phi_t&=\Bigl(\vec{0},1,\sum\limits_{i=1}^{n-1} x^i\dot{z}_i^1(t),\ldots,\sum\limits_{i=1}^{n-1} x^i\dot{z}_i^m(t)\Bigr).
\end{align*}
Therefore, the induced metric in the coordinates $(x^1,\ldots, x^{n-1}, t)$ is given by 
\begin{equation}\label{eqn:g}
    g= \begin{bmatrix}
        I_{(n-1)\times (n-1)}+\sum\limits_{\beta=1}^m\vec{z}^{\ \beta}(\vec{z}^{\ \beta})^\top & \vrule & \sum\limits_{\beta=1}^m \bigl(\langle \vec{x}, \dot{\vec{z}}^{\ \beta}\rangle \bigr) \vec{z}^{\ \beta}\\[1.5em]
        \hline
        & \vrule & \\
        \sum\limits_{\beta=1}^m \bigl(\langle \vec{x}, \dot{\vec{z}}^{\ \beta}\rangle \bigr) (\vec{z}^{\ \beta})^\top & \vrule & 1+\sum\limits_{\beta=1}^m \bigl(\langle \vec{x}, \dot{\vec{z}}^{\ \beta}\rangle \bigr)^2
    \end{bmatrix}
\end{equation}
where we abbreviate $\vec{x} = (x^1,\ldots,x^{n-1})^\top$, and $\vec{z}^{\ \alpha}=(z_1^\alpha,\ldots,z_{n-1}^\alpha)^\top$, $\alpha=1,\ldots,m$.

The (Euclidean) Hessian of the coordinate function $f^\alpha$, for a fixed $\alpha=1,\ldots,m$ and with respect to the coordinates $(x^1, \ldots, x^{n-1}, t)$, is given by
\begin{equation}\label{eqn:hessian}
    \Hess f^\alpha=\begin{bmatrix}
        0_{(n-1)\times (n-1)} & \vrule & \dot{\vec{z}}^{\ \alpha} \\[1em]
        \hline
        & \vrule & \\
        (\dot{\vec{z}}^{\ \alpha})^\top & \vrule & \langle \vec{x}, \ddot{\vec{z}}^{\ \alpha}\rangle 
    \end{bmatrix}.
\end{equation}
By Proposition \ref{prop:minimal-basic}, the Laplacian of $f^\alpha$ with respect to the induced metric $g$ is given by contracting the inverse of the matrix \eqref{eqn:g} and the matrix \eqref{eqn:hessian}. 
Due to the vanishing of the top-left block of $\Hess f^\alpha$, the Cramer's rule calculation is greatly simplified, and we obtain:   
\begin{equation}\label{eqn:Lap-fa-C}
    \Delta_g f^\alpha=2\sum_{i=1}^{n-1} C_{i,n} \dot{z}^\alpha_i+C_{n,n}\sum\limits_{i=1}^{n-1} x^i \ddot{z}^\alpha_i(t) ,
\end{equation}
where $C_{i,n}$ is the $(i,n)$-adjoint of $g$ and $C_{n,n}$ is the $(n,n)$-adjoint of $g$. Therefore, we deduce from \eqref{eqn:Lap-fa-C} that
\begin{align} 
    \Delta_g f^\alpha &= \det \begin{bmatrix}
        I_{(n-1)\times (n-1)}+\sum\limits_{\beta=1}^m\vec{z}^{\ \beta}(\vec{z}^{\ \beta})^\top & \vrule & 2\dot{\vec{z}}^{\ \alpha}\\[1.5em]
        \hline
        & \vrule & \\
        \sum\limits_{\beta=1}^m \bigl(\langle \vec{x}, \dot{\vec{z}}^{\ \beta}\rangle \bigr) (\vec{z}^{\ \beta})^\top & \vrule & \langle \vec{x}, \ddot{\vec{z}}^{\ \alpha}\rangle 
    \end{bmatrix} \notag \\
    &= \sum_k x^k\cdot\det \begin{bmatrix}
        I_{(n-1)\times (n-1)}+\sum\limits_{\beta=1}^m\vec{z}^{\ \beta}(\vec{z}^{\ \beta})^\top & \vrule & 2\dot{\vec{z}}^{\ \alpha}\\[1.5em]
        \hline
        & \vrule & \\
        \sum\limits_{\beta=1}^m \dot{z}_k^{\ \beta} (\vec{z}^{\ \beta})^\top & \vrule & \ddot{z}_k^{\ \alpha} 
    \end{bmatrix} . \label{eqn:Lap-fa}
\end{align}

If $\Phi(\vec{x},t)$ defines a minimal embedding, Proposition \ref{prop:minimal-basic} asserts that \(\Delta_g f^\alpha = 0\) for $\alpha = 1,\ldots,m$.  It follows that all the linear coefficients on the right hand side of \eqref{eqn:Lap-fa} must vanish, which implies the following equation:
\begin{align}\label{eqn:mss-gen-z}
    \det \begin{bmatrix}
        I_{(n-1)\times (n-1)}+\sum\limits_{\beta=1}^m\vec{z}^{\ \beta}(\vec{z}^{\ \beta})^\top & \vrule & 2\dot{\vec{z}}^{\ \alpha}\\[1.5em]
        \hline
        & \vrule & \\
        \sum\limits_{\beta=1}^m \dot{z}_k^{\ \beta} (\vec{z}^{\ \beta})^\top & \vrule & \ddot{z}_k^{\ \alpha} 
        \end{bmatrix}=0 ,
\end{align}
for \(k=1,\ldots,n-1\) and \(\alpha=1,\ldots,m\).

We recall an elementary determinant formula for a block matrix. 
\begin{lemma}\label{determinant}
    For a block matrix $\begin{bmatrix} A  & B  \\ C  & D  \end{bmatrix}$ with $A $ invertible and $D $ a scalar:
    \[ \det\begin{bmatrix} A  & B  \\ C  & D  \end{bmatrix} = \det(A ) \cdot (D  - C  A ^{-1} B ). \]
\end{lemma}

Therefore, the above determinant is zero if and only if $D  - C  A ^{-1} B =0$.  By taking $A =I+ZZ^\top $ and $D =\ddot{z}_k^\alpha$, we obtain \eqref{eqn:slope_evolution}.

We remark that Machado and Ferreira relate ruled minimal submanifolds and geodesic equations on affine Grassmannians in \cite{MF1996}.

\section{Grassmannian geodesic}\label{sec:geodesic}

In this section, we prove some lemmas concerning the Grassmannian geodesic equation \eqref{eqn:slope_evolution}. We start by reviewing Grassmannian geometry from two models and the geodesic equation in each model. 

\subsection{Grassmannian geodesic in the Stiefel model}
The Grassmannian $\Gr(n-1,n-1+m)$ consists of all $(n-1)$-dimensional linear subspaces in $\mathbb{R}^{n-1+m}$ and is isomorphic to $\Gr(m,n-1+m)$. It can be realized as the quotient of the Stiefel manifold $\St_m(\mathbb{R}^{n-1+m})$---the space of orthonormal $m$-frames in $\mathbb{R}^{n-1+m}$—by the right action of the orthogonal group $O(m)$. Namely,
\[
    \Gr(n-1,n-1+m) = \St_m(\mathbb{R}^{n-1+m}) / O(m),
\]
where
\[
    \St_m(\mathbb{R}^{n-1+m}) := \{ V \in \mathbb{R}^{(n-1+m) \times m} \mid V^\top V = I_m \}.
\]
The canonical Riemannian metric is induced from the Euclidean inner product $\langle \delta V_1, \delta V_2 \rangle := \tr(\delta V_1^\top \delta V_2)$ restricted to the horizontal bundle defined by $V^\top \delta V = 0$. Geodesics in $\Gr(n,n+m)$ then correspond to horizontal curves in the Stiefel manifold that are critical points of the energy functional
\[
\mathcal{E}[V] := \frac{1}{2} \int \tr(\dot{V}^\top \dot{V})\, \mathrm{d}t.
\]
See \cite{EAS1998} for more details.

\begin{definition}\label{3.1}
    A smooth family of $(n-1+m)\times m$ matrix $ V(t) = \begin{bmatrix} P(t) \\ Q(t) \end{bmatrix} $, where $P(t)\in \mathbb{R}^{m\times m}$ and $Q(t)\in \mathbb{R}^{(n-1)\times m}$, is a Grassmannian geodesic in Stiefel model if
    \begin{equation}\label{eq:geod-St}
        \begin{split}
V^\top V = I_m, \quad V^\top \dot{V} = 0, \quad \ddot{V} + V(\dot{V}^\top \dot{V}) = 0.            
        \end{split}
    \end{equation}
\end{definition}

\begin{proposition} \label{3.2} Suppose $V(t)$ is a Grassmannian geodesic in Stiefel model and $L$ is an element of $O(n-1+m)$, then $LV(t)$ is also a Grassmannian geodesic in Stiefel model. 
 
\end{proposition}

\begin{proof}
It is not hard to see that $\widetilde{V}(t)=LV(t)$ satisfies the conditions:
\[
    \widetilde{V}^\top \widetilde{V} = I_m, \quad \widetilde{V}^\top \dot{\widetilde{V}} = 0, \quad \ddot{\widetilde{V}} + \widetilde{V}(\dot{\widetilde{V}}^\top \dot{\widetilde{V}}) = 0,
\]
and the proposition follows.
\end{proof}

\subsection{Affine coordinates of the Grassmannian}

The Grassmannian $\Gr(n-1,n-1+m)$ consists of all $(n-1)$-dimensional linear subspaces in $\mathbb{R}^{n-1+m}$ and is canonically isomorphic to $\Gr(m,n-1+m)$ via orthogonal complement. A natural local parameterization is given by the \emph{affine coordinates}: identify a subspace that projects isomorphically onto $\mathbb{R}^m$ with the graph of a linear map $Z : \mathbb{R}^m \to \mathbb{R}^{n-1}$, represented as
\[V(Z) := \begin{bmatrix} I_m \\ Z \end{bmatrix}, \quad Z \in \mathbb{R}^{(n-1) \times m}.\]
This describes a smooth embedding of an open subset of $\mathbb{R}^{(n-1) \times m}$ into $\Gr(m,n-1+m)$, and hence into $\Gr(n-1,n-1+m)$. 

The Riemannian structure induced from the canonical metric on the Stiefel manifold leads to an explicit formula for the energy of curves $Z(t)$ in the affine coordinates. The geodesic equation in this chart, derived variationally by Wong \cite{Wong}, captures the intrinsic geometry of $\mathrm{Gr}(n-1,n-1+m)$ in terms of the coordinate matrix $Z(t)$.

\begin{definition}
A smooth curve of matrices $Z(t) \in \mathbb{R}^{(n-1) \times m}$ is a \emph{Grassmannian geodesic in affine coordinates} if it satisfies the second-order differential equation:
\[
\ddot{Z} = 2\dot{Z}Z^\top (I_{n-1} + ZZ^\top )^{-1} \dot{Z}.
\]
\end{definition}

To relate the affine coordinates to the Stiefel model, we observe that any smooth orthonormal frame $V(t) = \begin{bmatrix} P(t) \\ Q(t) \end{bmatrix} \in \mathbb{R}^{(n-1+m) \times m}$ with $V^\top V = I_{m}$ and $V^\top \dot{V} = 0$ projects to a Grassmannian curve.

\begin{lemma}\label{lem:st-to-graph}
Let $V(t) = \begin{bmatrix} P(t) \\ Q(t) \end{bmatrix}$ be a smooth curve in the Stiefel manifold $\St_m(\mathbb{R}^{n-1+m})$ such that $V^\top V = I_m$ and $V^\top \dot{V} = 0$. If $P(t)$ is invertible, then the corresponding curve in the associated graph coordinate, \(Z(t) := Q(t) P(t)^{-1}\),
satisfies
\[
\ddot{Z} - 2\dot{Z}Z^\top (I_{n-1} + ZZ^\top )^{-1} \dot{Z} = \left( \ddot{Q} - QP^{-1} \ddot{P} \right) P^{-1}.
\]
\end{lemma}

\begin{proof}
The Grassmannian conditions imply:
\begin{equation}\label{eq:Gr-condition}
    \begin{split}
 V^\top V &= P^\top P + Q^\top Q = I_m, \\
 V^\top \dot{V} &= P^\top \dot{P} + Q^\top \dot{Q} = 0.        
    \end{split}
\end{equation}
We compute via \eqref{eq:Gr-condition}:
 \begin{align*}
 I_m+Z^\top Z&=I_m+P^{-\top}Q^\top QP^{-1}=P^{-\top}(P^\top P+Q^\top Q)P^{-1}=P^{-\top}P^{-1},\\
 (I_m+Z^\top Z)^{-1}Z^\top &=P P^\top P^{-\top}Q^\top =PQ^\top ,
 \end{align*}
Differentiating $ Z = Q P^{-1} $, we obtain:
\[
    \dot{Z} = (\dot{Q} - Q P^{-1} \dot{P}) P^{-1}.
\]
Using these and the identity $Z^\top (I_{n-1} + ZZ^\top )^{-1}=(I_m + Z^\top Z)^{-1}Z^\top $, we compute:
\begin{align*}
    &\quad \dot{Z} Z^\top (I_{n-1} + ZZ^\top )^{-1} \dot{Z} \\
    &= \dot{Z} (I_m + Z^\top Z)^{-1}Z^\top \dot{Z}\\
    &= (\dot{Q} - Q P^{-1} \dot{P}) P^{-1} PQ^\top (\dot{Q} - Q P^{-1} \dot{P}) P^{-1}\\
    &= (\dot{Q} - Q P^{-1} \dot{P}) Q^\top (\dot{Q} - Q P^{-1} \dot{P}) P^{-1}.
\end{align*}
By the Grassmannian conditions \eqref{eq:Gr-condition} again, we have
\begin{align*}
    Q^\top (\dot{Q} - Q P^{-1} \dot{P})&=Q^\top \dot{Q}-(I_m-P^\top P)P^{-1}\dot{P}\\
    &= Q^\top \dot{Q}+P^\top \dot{P}-P^{-1}\dot{P}\\
    &=-P^{-1}\dot{P},
\end{align*}
which implies
\begin{align*}
    \dot{Z} Z^\top (I_{n-1} + ZZ^\top )^{-1} \dot{Z} =-(\dot{Q} - Q P^{-1} \dot{P})P^{-1}\dot{P}P^{-1}.
\end{align*}
On the other hand, differentiating $\dot{Z}$ again yields:
\begin{align*}
\ddot{Z} &= \ddot{Q} P^{-1} - 2\dot{Q} P^{-1} \dot{P} P^{-1} 
+ 2Q P^{-1} \dot{P} P^{-1} \dot{P} P^{-1} - Q P^{-1} \ddot{P} P^{-1} \\
&= (\ddot{Q} - Q P^{-1} \ddot{P}) P^{-1} - 2(\dot{Q} - Q P^{-1} \dot{P}) P^{-1} \dot{P} P^{-1}.
\end{align*}
This gives the claimed identity.
\end{proof}

\begin{proposition} \label{3.5}
Let $ V(t) = \begin{bmatrix} P(t) \\ Q(t) \end{bmatrix} \in \St_m(\mathbb{R}^{n-1+m})$ be a Grassmannian geodesic in Stiefel model as in Definition \ref{3.1}. 
Then $ Z(t) = Q(t) P(t)^{-1}\in \mathbb{R}^{(n-1)\times m} $ satisfies the Grassmannian geodesic equation in affine coordinates:
\[
\ddot{Z} = 2\dot{Z}Z^\top (I_{n-1} + ZZ^\top )^{-1} \dot{Z}.
\]
\end{proposition}

\begin{proof}
From the geodesic equation \eqref{eq:geod-St}, we have:
\begin{align*}
\ddot{P} + P(\dot{P}^\top \dot{P} + \dot{Q}^\top \dot{Q}) &= 0, \\
\ddot{Q} + Q(\dot{P}^\top \dot{P} + \dot{Q}^\top \dot{Q}) &= 0.
\end{align*}
Hence,
\[
\ddot{Q} - Q P^{-1} \ddot{P} 
= -Q(\dot{P}^\top \dot{P} + \dot{Q}^\top \dot{Q}) + Q P^{-1} P(\dot{P}^\top \dot{P} + \dot{Q}^\top \dot{Q}) = 0.
\]
Applying Lemma~\ref{lem:st-to-graph}, the result follows.
\end{proof}

\subsection{Transformation of geodesic equation under the \texorpdfstring{$O(n-1+m)$}{} action}

We now show how the geodesic equation transforms under the action of $O(n-1+m)$.

\begin{proposition}
Let $ Z(t) \in \mathbb{R}^{(n-1) \times m} $ be a curve in $ \Gr(n-1, n-1 + m) $, and let
\[
    \begin{bmatrix} A & B \\ C & D \end{bmatrix} \in \mathrm{O}(n-1 + m),
\]
where $A\in \mathbb{R}^{m\times m},\ B\in \mathbb{R}^{m\times (n-1)},\ C\in \mathbb{R}^{(n-1)\times m},\ D\in \mathbb{R}^{(n-1)\times (n-1)}$. Define
\begin{equation}\label{eqn:Mob}
 W(t) = (C + D Z(t))(A + B Z(t))^{-1}\in \mathbb{R}^{(n-1) \times m}.
\end{equation}
Then, as long as $(A+BZ)^{-1}$ exists, we have
\begin{align*}
    &\quad \ddot{W} - 2\dot{W} W^\top (I_{n-1} + W W^\top )^{-1} \dot{W} \\
    &= (D - W B) \left[ \ddot{Z} - 2\dot{Z} Z^\top (I_{n-1} + ZZ^\top )^{-1} \dot{Z} \right](A + B Z)^{-1}.
\end{align*}
In particular, Grassmannian geodesics in affine (graph) chart are preserved by transformation \eqref{eqn:Mob}. 
\end{proposition}

\begin{proof}
Let $ \begin{bmatrix} P \\ Q \end{bmatrix} $ be a Stiefel lift of the curve \( Z = QP^{-1} \), so that \( V = \begin{bmatrix} P \\ Q \end{bmatrix} \in \St_m(\mathbb{R}^{n-1+m}) \). Consider the transformed frame
\[
    \widetilde{V} := \begin{bmatrix} \widetilde{P} \\ \widetilde{Q} \end{bmatrix} := \begin{bmatrix} A & B \\ C & D \end{bmatrix} \begin{bmatrix} P \\ Q \end{bmatrix},
\]
with corresponding graph coordinate \( W = \widetilde{Q} \widetilde{P}^{-1} \), which is well-defined precisely on the affine coordinate patch where $\widetilde{P}$ is invertible (equivalently $A + BZ$ is invertible).

We emphasize that the orthogonality condition \( \begin{bmatrix} A & B \\ C & D \end{bmatrix} \in \mathrm{O}(n-1+m) \) is essential: it ensures that \( \widetilde{V} \) remains a Stiefel frame, so the correspondence remains valid for both \( Z \) and \( W \).

Using Lemma~\ref{lem:st-to-graph}, we write:
\begin{align*}
\ddot{Z} - 2\dot{Z} Z^\top (I_{n-1} + ZZ^\top )^{-1} \dot{Z} &= (\ddot{Q} - Z\ddot{P}) P^{-1}, \\
\ddot{W} - 2\dot{W} W^\top (I_{n-1} + WW^\top )^{-1} \dot{W} &= (\ddot{\widetilde{Q}} - W \ddot{\widetilde{P}}) \widetilde{P}^{-1}.
\end{align*}
Now compute:
\begin{align*}
    \ddot{\widetilde{Q}} \widetilde{P}^{-1} - W \ddot{\widetilde{P}} \widetilde{P}^{-1}
    &= (C \ddot{P} + D \ddot{Q})(A P + B Q)^{-1} - W (A \ddot{P} + B \ddot{Q})(A P + B Q)^{-1} \\
    &= (D - W B) \ddot{Q} P^{-1} (A + B Z)^{-1} + (C - W A) \ddot{P} P^{-1} (A + B Z)^{-1}.
\end{align*}
Using the identity \( C + D Z = W(A + B Z) \), we rearrange to get \( C - W A = - (D - W B) Z \), and thus
\[
    \ddot{\widetilde{Q}} \widetilde{P}^{-1} - W \ddot{\widetilde{P}} \widetilde{P}^{-1}
    = (D - W B)(\ddot{Q} - Z \ddot{P}) P^{-1} (A + B Z)^{-1}.
\]
This proves the claimed transformation law.
\end{proof}

\section{Proof of Theorem \ref{thm:entire_ODE}}\label{sec:entire_ODE}
We recall that $(B, \widetilde{\Lambda})$ satisfy that 
\[ \det\bigl(\cos(\widetilde{\Lambda} t)+B^\top \sin(\widetilde{\Lambda} t)\,\bigr) > 0 \]
for all \(t\). We claim that 
\begin{align}\label{Z-solution}
    Z(t) &=
    \bigl(-NB\cos(\widetilde{\Lambda} t) + N\sin(\widetilde{\Lambda} t)\bigr)\,
    \bigl(M\cos(\widetilde{\Lambda} t) + MB^\top \sin(\widetilde{\Lambda} t)\bigr)^{-1} 
\end{align}
is an entire solution of \eqref{eqn:slope_evolution} with the given initial conditions where
\[
    M = (I + B^\top B)^{-1/2},\quad N = (I + B B^\top )^{-1/2}.
\]
The initial conditions of $Z(t)$ in \eqref{Z-solution} are computed as follows: 
\[
    Z(0) = -\,NB\,M^{-1}
    = -\,(I + B B^\top )^{-1/2}\,B\,(I + B^\top B)^{1/2}.
\]
Using the singular value decomposition of $B$, one can check that $NB=BM$, and thus $Z(0)=-B$.

Rewriting \eqref{Z-solution} as
\[
    Z(t)\,M\,\bigl(\cos(\widetilde{\Lambda} t)+B^\top \sin(\widetilde{\Lambda} t)\bigr)
    =N\,\bigl(-B\cos(\widetilde{\Lambda} t)+\sin(\widetilde{\Lambda} t)\bigr),
\]
and differentiating at $t=0$ gives
\[
    \dot{Z}(0)\,M + Z(0)\,M B^\top \widetilde{\Lambda}
    = N\,\widetilde{\Lambda}.
\]
Using $Z(0)=-NBM^{-1}$ and $NB=BM$, we obtain
\[
    \dot{Z}(0)\,M = N (I+BB^\top)\,\widetilde{\Lambda},
    \qquad\text{i.e.}\qquad
    \dot{Z}(0)
    = N (I+BB^\top)\,\widetilde{\Lambda}\,M^{-1},
\]
which reduces to
\[
    \dot{Z}(0)
    = (I+BB^\top)^{1/2}\,\widetilde{\Lambda}\,(I+B^\top B)^{1/2}.
\]
The curve $Z(t)$ is clearly defined for all $t\in (-\infty, \infty)$. Next we show that $Z(t)$ defined by \eqref{Z-solution} is a solution of \eqref{eqn:slope_evolution}.
Let 
\[
    P = M\cos(\widetilde{\Lambda} t) + M B^\top \sin(\widetilde{\Lambda} t) \quad\text{ and }\quad
    Q = -NB\cos(\widetilde{\Lambda} t) + N\sin(\widetilde{\Lambda} t) .
\]
The size of $P$ is $m\times m$, and the size of $Q$ is $ (n-1)\times m$.  Note that $Z=QP^{-1}$.

Write 
\[ \begin{bmatrix} P(t)\\Q(t)\end{bmatrix}
=\begin{bmatrix}M & MB^\top \\-NB & N\end{bmatrix}
\begin{bmatrix}
\cos(\widetilde{\Lambda} t) \\
\sin(\widetilde{\Lambda} t) 
\end{bmatrix}.
\]
It is straightforward to check that the matrix $\begin{bmatrix}M & MB^\top \\-NB & N\end{bmatrix}
$ is an element of $O(n-1+m)$ and $\begin{bmatrix}
\cos\widetilde{\Lambda} t \\
\sin\widetilde{\Lambda} t 
\end{bmatrix}$ is a Grassmannian geodesic in the Stiefel model. Therefore, $\begin{bmatrix} P(t)\\Q(t)\end{bmatrix}$ is a Grassmannian geodesic in the Stiefel model by Proposition \ref{3.2} and $Z(t)=Q P^{-1}$ is a Grassmannian geodesic in  affine coordinates by Proposition \ref{3.5}.

\section{Proof of Theorem \ref{thm:entire_odd}}\label{sec:entire_odd}

By Theorem \ref{thm:entire_ODE}, it suffices to construct explicit pairs of $(n-1)\times m$ matrices 
$\widetilde{\Lambda}$ and $B$, such that 
\[
\det\bigl(\cos(\widetilde{\Lambda} t)+B^\top\sin(\widetilde{\Lambda} t)\bigr)
>0
\quad\text{for all }t ,
\]
when $n$ is odd.

In the special case where $n-1=m$ and $\widetilde{\Lambda}=\Lambda = I_m$ (so $m$ is even), the condition becomes
$\det\bigl(\cos t\,I_m+B^\top\sin t\bigr)>0$.  
When $\sin t=0$, this determinant is $1$.  
When $\sin t\neq 0$, it has the same sign as
$\det(\cot t\,I_m + B^\top)$, which is positive if $B$ has no real eigenvalues.  
Since $m$ is even, a real matrix with no real eigenvalues has characteristic polynomial positive on $\mathbb{R}$, hence the determinant is positive for all $t$.

We now generalize this to the case where $\widetilde{\Lambda}$ and $B$ are block–diagonal.  
Suppose there exist distinct real numbers $\widetilde{\lambda}_i$, $i=1,\ldots,k$, and even integers $d_i>0$ with
\[
\sum_{i=1}^k d_i = s \le \min(n-1,m)
\]
such that
\[
\Lambda =
\begin{bmatrix}
\widetilde{\lambda}_1 I_{d_1} & & \\
 & \ddots & \\
 & & \widetilde{\lambda}_k I_{d_k}
\end{bmatrix}, \qquad
B =
\begin{bmatrix}
B_1 & & \\
 & \ddots & \\
 & & B_k
\end{bmatrix},
\]
where each $B_i$ is a $d_i\times d_i$ real matrix with no real eigenvalues (equivalently, $\det(x I_{d_i} - B_i)>0$ for all real $x$ since $d_i$ is even).  
Embedding this $\Lambda$ as the top-left $s\times s$ block of $\widetilde{\Lambda}$ (and extending each $B_i$ by zeros to a block of the same location in $B$) yields a pair $(\widetilde{\Lambda},B)$ satisfying the desired positivity condition for all $t$.

\section{Entire special Lagrangian graphs and proof of Theorem \ref{thm:entire_SPL}}\label{sec:entire_SPL}
It suffices to show that the entire minimal graphs constructed in Theorem \ref{thm:entire_odd} are austere and then apply Harvey-Lawson's Theorem. We recall that a submanifold of Euclidean space is said to be {\it austere} in the sense of Harvey-Lawson \cite{HL82} if, for every normal direction $\nu$, the $k$-th elementary symmetric function of the eigenvalues of the second fundamental form in the direction of $\nu$ vanishes for all odd positive integers $k$ no greater the dimension of the submanifold. The condition is stronger than minimality, which corresponds to the vanishing of the first elementary symmetric function ($k=1$). We show that, for any minimal graph defined by $f^\alpha$ under the evolving plane ansatz \eqref{eqn:form}, the second fundamental form in any normal direction has at most two nonzero eigenvalues of opposite signs. In fact, these submanifolds are {\it simple austere} in the sense of Bryant (see Section 3 of \cite{Bryant}). For a suitable basis $n^\alpha, \alpha=1, \cdots, m$ of the normal bundle, the second fundamental form corresponding to $n^\alpha$ is precisely $\Hess f^\alpha$ in equation  \eqref{eqn:hessian}. Consequently, the second fundamental form in any normal direction is represented by a $n \times n $ matrix whose upper-left $(n-1)\times (n-1)$ block vanishes. Such a matrix has rank at most two. The minimality condition implies that its trace is zero, and therefore the two (possibly) non-zero eigenvalues sum to zero. Hence, the minimal graph is austere.

\appendix

\section{Lorentzian (Semi-Riemannian) setting}\label{app:Lorentzian}

In the appendix, we extend our construction to the Lorentzian (semi-Riemannian) setting. A spacelike graph over $\mathbb{R}^n$
in the Minkowski space $\mathbb{R}^{n,1}$
 with vanishing mean curvature is called a maximal graph. A classical theorem of Cheng and Yau \cite{CY} shows that there are no entire maximal graphs in any dimension, in sharp contrast to the Riemannian case, where the well-known Bernstein theorem holds up to dimension seven.

We construct graphs over a Minkowski base $\mathbb{R}^{n-1, 1}$ with Lorentzian induced metric (of signature $(n-1, 1)$) and vanishing mean curvature in the ambient semi-Riemannian space $\mathbb{R}^{n-1, m+1}$. More precisely, we consider graphs defined over $\mathbb{R}^{n-1,1}$ in the ambient $\mathbb{R}^{n-1,m+1}\cong\mathbb{R}^{n+m}$ endowed with metric
\[
    \langle (x,t,y),(x,t,y)\rangle = |x|^2 - t^2 - |y|^2,
    \qquad (x,t,y)\in\mathbb{R}^{n-1}\times\mathbb{R}\times\mathbb{R}^m.
\]

The coordinates $(x^1,\dots,x^{n-1},t)$ are used on the domain.

\begin{definition}[Evolving–plane ansatz, semi–Euclidean]
Let $\{z_i^\alpha(t)\}_{\alpha=1,\ldots,m;\ i=1,\ldots,n-1}$ be smooth functions. Define
\begin{equation}\label{eqn:form-L}
 f^\alpha(x,t)=\sum_{i=1}^{n-1} z_i^\alpha(t)\,x^i,
\end{equation}
and set $\vec{z}^{\,\alpha}(t)=(z_1^\alpha(t),\ldots,z_{n-1}^\alpha(t))^\top\in\mathbb{R}^{n-1}$.
\end{definition}

Similarly, the induced metric on the graph embedding
\[
    \Phi(x,t)=\bigl(x^1,\dots,x^{n-1},t,f^1(x,t),\dots,f^m(x,t)\bigr)
\]
has coefficients as follows:
\begin{equation}\label{eqn:g_L}
g=\begin{bmatrix}
I_{(n-1)\times(n-1)}-\sum_{\beta=1}^m \vec{z}^{\,\beta}(\vec{z}^{\,\beta})^\top
& \vrule & -\sum_{\beta=1}^m \langle \vec{x},\dot{\vec{z}}^{\,\beta}\rangle \,\vec{z}^{\,\beta}\\[1.5em]
\hline
 & \vrule & \\
-\sum_{\beta=1}^m \langle \vec{x},\dot{\vec{z}}^{\,\beta}\rangle (\vec{z}^{\,\beta})^\top
& \vrule & -1-\sum_{\beta=1}^m\bigl(\langle \vec{x},\dot{\vec{z}}^{\,\beta}\rangle\bigr)^2
\end{bmatrix}.
\end{equation}

\begin{theorem}[$H=0$ over $\mathbb{R}^{n-1,1}$ $\Leftrightarrow$ pseudo–Grassmannian geodesic]
If $f^\alpha$ is of the form \eqref{eqn:form-L}, then the image of $\Phi$ has vanishing mean curvature in $\mathbb{R}^{n-1,m+1}$ if and only if
\begin{equation}\label{eqn:slope_evolution_L}
    \det \begin{bmatrix}
    I_{(n-1)\times(n-1)}-\sum_{\beta=1}^m \vec{z}^{\,\beta}(\vec{z}^{\,\beta})^\top 
    & \vrule & 2\dot{\vec{z}}^{\,\alpha}\\[1.5em]
    \hline
    & \vrule & \\
    \sum_{\beta=1}^m \dot z_k^{\,\beta}(\vec{z}^{\,\beta})^\top 
    & \vrule & \ddot z_k^{\,\alpha}
\end{bmatrix}=0
\end{equation}
for \(k=1,\ldots,n-1\) and \(\alpha=1,\ldots,m\).
Equivalently, the matrix form is the follows:
\begin{equation}\label{ODE_L}
    \ddot Z+2\,\dot Z Z^\top (I_{n-1}-ZZ^\top)^{-1}\dot Z=0,
\end{equation}
where $Z=[z_i^\alpha]\in\mathbb{R}^{(n-1)\times m}$.
\end{theorem}

\begin{remark}
Equation \eqref{ODE_L} is the geodesic equation on the pseudo–Grassmannian $\Gr^+_{n-1}(n-1,m+1)$ in affine coordinates. This is the pseudo-Riemannian analogue of \eqref{eqn:slope_evolution}, obtained by replacing $I+ZZ^\top$ with $I-ZZ^\top$.
\end{remark}
Consider the following initial conditions for solutions to \eqref{ODE_L} 
\[
    Z(0) = O_{(n-1) \times m}, \qquad 
    \dot{Z}(0) = 
    \begin{bmatrix}
        \Lambda &0\\[2pt]
        0&0
    \end{bmatrix}.
\]
for $\Lambda=\diag(\lambda_1\ldots\lambda_r)$, $0\le r\le\min(n-1,m)$ and $0<\lambda_1\le\cdots\le\lambda_r$. 

Denoting $\tanh(\Lambda t)=\diag(\tanh(\lambda_1 t),\ldots,\tanh(\lambda_r t))$, the corresponding explicit solution to \eqref{ODE_L} is 
\[
Z(t) = 
\begin{bmatrix}
\tanh(\Lambda t)& 0\\0& 0
\end{bmatrix}.
\]

\begin{proposition}[Global Lorentzian property of the $\tanh$ family]
For $Z(t)$ as above, the induced metric of $\Phi$ has signature $(+,\cdots,+,-)$ for all $(x,t)$.
\end{proposition}

\begin{proof}
For $1\le i\le r$, $z_i^{\,i}=\tanh(\lambda_i t)$ and $\dot{z}_i^{\,i}=\lambda_i\sech^2(\lambda_i t)$; all other $z_i^\alpha$ vanish.  It follows that
\[
    \sum_{\beta=1}^m \vec{z}^{\,\beta}(\vec{z}^{\,\beta})^\top
    = \diag(\tanh^2(\lambda_1 t),\dots,\tanh^2(\lambda_r t),0,\dots,0),
\]
and hence the top-left block of $g=\begin{bmatrix} A & B \\ C & D \end{bmatrix}$ in \eqref{eqn:g_L} is the positive definite matrix
\[
    A=\diag(\sech^2(\lambda_1 t),\dots,\sech^2(\lambda_r t),1,\dots,1).
\]

By using $\sech^2+\tanh^2=1$, the corresponding $D - C A^{-1} B$ of $g$ is
\[
    -1-\sum_{i=1}^r \lambda_i^2 (x^i)^2 \sech^2(\lambda_i t) < 0.
\]
Thus, $A$ has $n-1$ positive eigenvalues, $D - C A^{-1} B$ is negative, and $g$ has Lorentzian signature globally. Moreover,
\[
\det g = -\Big(\prod_{i=1}^r \sech^2(\lambda_i t)\Big)\Big(1+\sum_{i=1}^r \lambda_i^2(x^i)^2\sech^2(\lambda_i t)\Big)<0
\]
confirms that the induced metric is Lorentzian of index $1$.
\end{proof}

\end{document}